\documentclass[12pt,a4paper,twoside]{article}
\usepackage[english]{babel}
\usepackage{csquotes}
\usepackage{mathrsfs,mathtools,amsmath, amsthm,amssymb, amscd, amsfonts}
\usepackage{thmtools}  
\usepackage{setspace}
\usepackage{graphicx}
\usepackage{tabularx}
\usepackage[table]{xcolor}
\usepackage{ltablex}
\usepackage{enumitem,kantlipsum}
\usepackage{bbm}
\usepackage{longtable}
\usepackage{pdfpages}
\setlist[itemize]{noitemsep}
\setlist[enumerate]{itemsep=0mm}
\usepackage[top=3cm,bottom=3cm, left=3cm,right=3cm,footskip=1cm]{geometry}
\usepackage{appendix}
\usepackage{tikz}
\usepackage{tikz-cd}
\usepackage{xcolor}
\usepackage{adjustbox}
\definecolor{persianblue}{rgb}{0.11, 0.22, 0.73}

\definecolor{persiangreen}{rgb}{0.0, 0.65, 0.58}
\definecolor{antiquefuchsia}{rgb}{0.57, 0.36, 0.51}

\theoremstyle{definition}
\newtheorem{theorem}{Theorem}[section]
\newtheorem{lemma}[theorem]{Lemma}

\newtheorem{question}[theorem]{Question}
\newtheorem{definition}[theorem]{Definition}

\newtheorem{proposition}[theorem]{Proposition}

\newtheorem{conjecture}[theorem]{Conjecture}

\newtheorem{namedlemma}{}

\usepackage{stackengine}
\stackMath

\usepackage[bookmarks=false,hidelinks]{hyperref}
\hypersetup{
	colorlinks   = true, 
	urlcolor     = antiquefuchsia, 
	linkcolor    = persianblue, 
	citecolor   = persianblue 
}
\usepackage{url}
\usepackage{nameref}
\usepackage[nameinlink,capitalize]{cleveref}
\makeatletter
\if@cref@capitalise
\Crefname{claim}{Claim}{Claims}
\Crefname{subclaim}{Sub-claim}{Sub-claims}
\Crefname{lemma}{Lemma}{Lemmas}
\Crefname{question}{Question}{Questions}
\Crefname{fact}{Fact}{Facts}
\Crefname{remark}{Remark}{Remarks}
\Crefname{observation}{Observation}{Observations}

\usepackage{accents}

\newcommand{\ad}{\textsf{AD}} 
\newcommand{\dc}{\textsf{DC}} 
\newcommand{\lsa}{\textsf{LSA}} 
\newcommand{\ac}{\textsf{AC}} 
\newcommand{\zf}{\textsf{ZF}} 

\newcommand{\cof}[1]{\mathrm{cof}(#1)} 
\newcommand{\rang}[1]{\mathrm{rang}(#1)} 
\newcommand{\rest}[1]{\restriction{#1}} 

\newcommand{\ord}{\mathrm{ORD}} 
\newcommand{\R}{\mathbb{R}} 
\newcommand{\N}[1]{\mathrm{N}} 

\newcommand{\df}{\coloneqq} 

\newcommand{\mbb}[1]{\mathbb{#1}} 
\newcommand{\mcl}[1]{\mathcal{#1}} 

\newcommand{\s}{\subseteq}

\usepackage[defernumbers=true,maxnames=6,
backend=biber,style=numeric,giveninits=true,sortcites=true,
sorting=nty,doi=true,isbn=true,url=false,safeinputenc]{biblatex}
\begin{document}
	
	\title{A Note on a Theorem of Apter}
	\author{\bigskip Rahman Mohammadpour, Otto Rajala, Sebastiano Thei \\ IMPAN\footnote{Institute of Mathematics of the Polish Academy of Sciences, ul. Śniadeckich 8, 00-656 Warszawa, Poland. Mohammadpour email: frng-mpr@proton.me, Rajala email: orajala@impan.pl, Thei email: sthei@impan.pl.}}
	\date{}
\maketitle
\begin{abstract}
We show that the consistency of $\zf+\ad_{\R}+`` \Theta$ \textsf{is measurable}" implies the consistency of $\zf+``\Theta$ \textsf{is the least strongly regular cardinal and the least measurable cardinal}"$+$ \textsf{$``$all uncountable cardinals below $\Theta$ are of countable cofinality.}"
\end{abstract}
\section{Introduction}
Large cardinal axioms are typically stated in terms of elementary embeddings from the universe $V$ into some transitive subclass $M$. The seminal idea is due to Scott who provided a reformulation of measurability, where a cardinal $\kappa$ is measurable if it carries a \emph{measure}, i.e., a $\kappa$-complete normal  non principal ultrafilter on $\kappa$. Specifically, he proved that a cardinal is measurable if and only if it's the critical point of an elementary embedding $j:V\rightarrow M$ into a transitive model $M$. However, in the absence of the axiom of choice ($\ac$) the equivalence may no longer hold. Remarkably, Jech \cite{Jech19681CB} showed in the 1960s that it is consistent for $\omega_1$ to be measurable.

A natural competitor of $\ac$ is the \emph{Axiom of Determinacy} ($\ad$) which asserts that every set of reals is determined, i.e., in every two-player game of length $\omega$ in which players play integers, a player has a winning strategy (here we treat $\omega^{\omega}$ as the set of reals.) Indeed, $\ad$ wipes out most of the usual patterns of cardinals inoculated by the non-constructive nature of $\ac$. For instance, a result due to Steel \cite[Theorem 8.27]{steel2009outline} asserts that, under $\ad+V=L(\mathbb{R})$, all regular cardinals below $\Theta$ are measurable, where $\Theta$ is the least cardinal onto which the real numbers $\mathbb{R}$ cannot be mapped. 

This configuration prompts the problem of identifying the smallest possible measurable cardinal in the absence of the axiom of choice. Although it is consistent that the first measurable cardinal be as small as $\omega_1$, this observation no longer yields an immediate answer once one imposes the additional requirement that the measurable cardinal be inaccessible. A first fruitful attempt has been provided by Apter \cite{Apter_AD_pattern}. He proved that, under $\ad+V=L(\mathbb{R})$, there is a model of $\zf$ where the least measurable cardinal is the least regular limit cardinal. More specifically, he exploited the Prikry forcing machinery within a $\zf$ framework to singularize any measurable cardinal contained in a given subset of $\Theta$, while preserving the measurability of all remaining measurable cardinals. Along the same line, Gitik, Hayut and Karagila \cite{gitik2024first} showed that if $\kappa$ is a $\kappa^{++}$-supercompact cardinal, then there is a forcing extension in which $\kappa$ is a measurable cardinal, which is also the least strongly inaccessible cardinal. We note that their definition of \emph{inaccessibility} is that a cardinal $\kappa$ is \emph{inaccessible} if it is uncountable and for every $x\in V_\kappa$ and every $f:x\to\kappa$, $\rang{f}$ is bounded. 

In this paper we are interested in reducing the large cardinal strength of the above theorem of Gitik, Hayut, and Karagila. However, our notion of \emph{inaccessibility} is slightly different, although both  are equivalent under $\ac$.  Nevertheless, both imply regularity. To  distinguish these two notions, we call ours \emph{strongly regular}.
A cardinal $\kappa$ is \emph{strongly regular} if it is uncountable and for every $\alpha<\kappa$ and every $f:\mcl P(\alpha)\to\kappa$, $\rang{f}$ is bounded in $\kappa$.  Recall that a measure $\mu$ is said to be \emph{$\mathbb{R}$-complete} if whenever $\langle A_x\mid x\in\mathbb{R}\rangle$ is a sequence of elements in $\mu$, $\bigcap_{x\in\mathbb{R}}A_x\in\mu$. Note that $\mbb R$-completeness implies  $\Theta$-completeness. We let $\Theta_{\sf meas}$ be the conjunction of $\ad_{\R}$ and ``\textsf{there is an $\mbb R$-complete measure on $\Theta$}''.  Recall that also $\ad_{\R}$ states that any $\omega$-length alternating two-player game with perfect information in which players play reals  is determined. 

Our main theorem reads as follows.
\begin{theorem}\label{mainthm}
Assume $\zf+ \Theta_{\sf meas}$. There is an extension satisfying $\zf$, in which $\Theta$ is strongly regular cardinal, and simultaneously, the least uncountable regular cardinal and the least measurable cardinal.
\end{theorem}
 It has been shown by Rachid and Sargsyan \cite{Atmai-Grigor} that the above assumption is consistent. On the other hand, $\zf+ \Theta_{\sf meas}$ is weaker than Woodin limit of Woodin cardinals (see \cite{Grigor-Nam}). 

The guiding expectation motivating the present work relies on the fact that the framework of \emph{Nairian models} \cite{blue2025nairian} should ultimately yield a construction of a model of $\zf$ in which the least measurable cardinal coincides with the least inaccessible cardinal in the sense of Gitik, Hayut and Karagila. More concretely, a suitable adaptation of the techniques employed in \cite{blue2026failure} to show that the failure of square at all uncountable cardinals is weaker than a Woodin limit of Woodin cardinals suggests the following conjecture.

\begin{conjecture}
    The theory ``$\zf+\Theta$ is the least measurable and the least inaccessible'' is weaker than a Woodin limit of Woodin cardinals.
\end{conjecture}

The organization of the paper is as follows. In the next section, we fix some notation and basics, and in the \cref{sec:main forcing}, we prove the main theorem while closely following Apter's analysis \cite{Apter_AD_pattern}. The last section is devoted to open questions.

\section{Basics}

  We follow standard notation. In particular, our forcing convention is that a condition $p\in\mbb P$ is stronger than $q\in\mbb P$ if $p\leq_{\mbb P} q$.
Let us now recall the definition of the Prikry forcing using a measure $\mu$. 
\begin{definition}
Let $\mu$ be a measure on $\kappa$.
The \emph{Prikry forcing} $\mathbb P_\mu$ consists of conditions $p\df(s_p,A_p)$, where
\begin{enumerate}
   \item $s_p\in [\kappa]^{<\omega}$, 
    \item $A_p\in\mu$, and
    \item $\max(s_p)<\min(A_p)$.
\end{enumerate}
A condition $p$ is stronger than $q$, i.e., $p\leq q$, if $s_p\supseteq s_q$, $A_p\subseteq A_q$, 
and  $s_p\setminus s_q\subseteq A_q$.
\end{definition}
Recall that the \emph{Prikry property} for $\mathbb{P}_\mu$ states that
    for every formula  $\varphi$ in the language of $\mathbb P_\mu$, and every $p=\langle s, A\rangle\in\mathbb{P}_\mu$, there is $B\in\mu$ with $B\s A$ such that $\langle s, B\rangle$ decides $\varphi$. 
\begin{theorem}[Apter \cite{Apter_AD_pattern}, \zf]\label{fact: Prikry_in_ZF}
     Let $\mu$ be a measure on $\kappa$. Then 
     \begin{enumerate}
         \item\label{item: PP} $\mathbb P_\mu$ has the Prikry property;
         \item\label{item: card_pres} $\mathbb P_\mu$ is cardinal preserving;
         \item $\Theta^{V^{\mathbb{P}_\mu}}=\Theta^{V}$.
     \end{enumerate}
\end{theorem}
\begin{proof}
    For (2) and (3) see, respectively, \cite[Lemma 1.3]{Apter_AD_pattern} (and its subsequent paragraph) and \cite[Lemma 1.4]{Apter_AD_pattern}. We verify the Prikry property. Accordingly, let $\varphi$ be a formula and fix $p=(s, A)\in\mathbb{P}_\mu$. Define a partition $h:[A]^{<\omega}\rightarrow 2$ as $h(t)=1$ if and only if there is $B\s A$ in $\mu$ with $\max(s\cup t)<\min(B)$ such that $(s\cup t, B)\Vdash\varphi$. Rowbottom's theorem holds, even without $\ac$ (see, e.g., \cite[Lemma 1.1]{Apter_AD_pattern}), and so there is $C\s A$ in $\mu$ homogeneous for $h$, i.e., for each $n<\omega$ and each $s_1,s_2\in [C]^n$, $h(s_1)=h(s_2)$. We claim that $(s, B)$ decides $\varphi$. Otherwise, there are conditions $(s\cup s_1, B_1), (s\cup s_2, B_2)\leq(s, C)$ with $|s_1|=|s_2|$ such that $(s\cup s_1, B_1)\Vdash\varphi$ and $(s\cup s_2, B_2)\Vdash\neg\varphi$. But $h(s_1)\neq h(s_2)$ contradicting the fact that $s_1,s_2\in [C]^{|s_1|}$ implies $h(s_1)= h(s_2)$.
\end{proof}
\begin{definition}
Let $\mcl F$ be a filter on a set $X$.
    Let $G$ be a $V$-generic filter on a forcing $\mbb P$. Working in $V[G]$, we define the \emph{lift} of $\mcl F$ in $V[G]$ by
    \[\mcl F^{*G}\df\{x\s X \mid\exists y\in\mcl F (y\s x)\}\]   
\end{definition}

\begin{lemma}[Apter \cite{Apter_AD_pattern}, \zf]\label{lifting_a_measure}  
    Assume $\kappa<\lambda$ are measurable cardinals with measures $\mu_\kappa$ and $\mu_\lambda$, respectively. Let $G\subseteq \mathbb{P}_{\mu_\kappa}$ be a $V$-generic filter.   Then $\mu^{\ast G}_\lambda$ is a measure on $\lambda$. Moreover, if $\mu_\lambda$ is normal, then $\mu^{\ast G}_\lambda$ is normal
\end{lemma}
\begin{proof}
    See \cite[Lemma 1.5]{Apter_AD_pattern}.
\end{proof}

\section{The main theorem}\label{sec:main forcing}
Throughout the rest of the paper, we assume $\zf+\Theta_{\sf meas}$.
A theorem due to Kunen  states that under $\ad$
if $\kappa<\Theta$ is measurable, then
the set of measures on $\kappa$ is well-ordered in a canonical way, see  \cite[Theorem 13.7]{Lar00bookExte} for a proof. More importantly, there is a uniform definition of well-ordering. So for each measurable cardinal $\kappa<\Theta$, let $\mu_\kappa$ be the $\prec_\kappa$-least measure on $\kappa$. Now let us set  \[
\mcl M\df\{\kappa<\Theta\mid \kappa\text{ is measurable}\}.
\]
Note that under $\ad_{\R}$ a cardinal $\kappa<\Theta$ is regular if and only if it is measurable.
We denote the Prikry forcing $\mbb P_{\mu_\kappa}$ by $\mbb P_\kappa$. Let
\[
\mathbb{P}\df\prod^{\text{Fin}}_{\kappa\in \mcl M}\mathbb{P}_\kappa
\]
be the finite support product of the Prikry forcings $\langle \mbb P_\kappa\mid\kappa\in \mcl M\rangle$.
Note that for every  nonempty  set $c\subseteq \mcl M$, there is a natural projection 
\begin{align*}
 \pi_c:\mbb P&\longrightarrow \mbb P_c\df  \prod_{\kappa\in c}\mbb P_\kappa\\
    p&\mapsto p\restriction c
\end{align*}

For a (generic) filter $G$ is a $\mathbb{P}$-generic filter over $V$, we let $G_c$ be the induced (generic) filter over $\mbb P_c$. We shall then denote $G_{\{\kappa\}}$ by $G_\kappa$, for each $\kappa\in A$. 
\begin{lemma}\label{lem: same cardinals}
    For each finite $c\s \mcl M$, $V$ and $V[G_c]$ have the same class of cardinals.
\end{lemma}
\begin{proof}
The proof is a simple adaptation of \cite[Lemma 2.2]{Apter_AD_pattern}.
Suppose $c=\{\kappa_0,\dots,\kappa_n\}$ with $\kappa_0<\cdots<\kappa_n$. By  \cref{fact: Prikry_in_ZF}\eqref{item: card_pres}, $V$ and $V[G_{\kappa_n}]$ have the same class of cardinals. Moreover, the Prikry property (\cref{fact: Prikry_in_ZF}\eqref{item: PP}) ensures that $\mathbb{P}_{\kappa_n}$ does not add bounded subset of $\kappa_{n}$. In particular, $V[G_{\kappa_n}]\models``\mu_{\kappa_{n-1}}\text{ is a measure on $\kappa_{n-1}$''}$. Working in $V[G_{\kappa_n}]$, we apply  \cref{fact: Prikry_in_ZF}\eqref{item: card_pres} to $\mathbb{P}_{\kappa_{n-1}}$, and deduce that $V$ and $V[G_{\kappa_{n-1}}\times G_{\kappa_n}]$ have the same class of cardinals. The same argument shows that the forcings $\mathbb{P}_{\{\kappa_{n-2}, \kappa_{n-1}, \kappa_n\}}, \dots, \mathbb{P}_{\{\kappa_0, \dots, \kappa_n\}}=\mathbb{P}_c$ are cardinal preserving. 
\end{proof}
Let $B\subseteq \mcl M$. For   a $V$-generic filter $G\subseteq \mbb P$. We shall define the   symmetric extension $N$ given by the the small generic filters $G_\kappa$'s, $\kappa\in \mcl M$, by induction, as follows. Let $\mathcal{L}_1$ be the sublanguage of the forcing language with respect to $\mathbb{P}_B$ which contains symbols $\dot{v}, \dot{G}_\kappa$ for each $v\in V$ and $\kappa\in \mcl M$ respectively, and an unary predicate symbol $\dot{V}$ for $V$. Stipulate $N_0=\emptyset$. For $\alpha\in\ord$, define 
\[
N_{\alpha+1}=\{x\subseteq N_\alpha\mid x \text{ is definable over $\langle N_\alpha, \in , \{c\mid c\in N_\alpha\}\rangle$ by a $\mathcal{L}_1$-term}\}.
\]
If $\lambda$ is a limit ordinal, $N_\lambda=\bigcup_{\alpha<\lambda}N_\alpha$. Finally, $N\df\bigcup_{\alpha\in\ord}N_\alpha$.
 
\begin{lemma}[Apter {\cite[Lemma 2.3]{Apter_AD_pattern}}]\label{lem: Theta computed correctly}
    $\Theta^N=\Theta^V$.
\end{lemma}

 \begin{lemma}\label{finite_piece_included}
     Suppose that $c\subseteq A$ is finite and nonempty. Then $V[G_c]\subseteq N$.
 \end{lemma}
 \begin{proof}
Suppose $c=\{\kappa_0,\dots,\kappa_n\}$ with $\kappa_0<\cdots<\kappa_n$. By construction, each $G_{\kappa_i}$ lives in $N$ and so $V[\{G_{\kappa_0},\dots, G_{\kappa_n}\}]\s N$. Since $V[\{G_{\kappa_0},\dots, G_{\kappa_n}\}]$ is trivially $V[G_c]$, the conclusion follows.    
 \end{proof}
The following key lemma is due to Apter.    
\begin{lemma}[Apter {\cite[Lemma 2.1]{Apter_AD_pattern}}]\label{lem: capturing}
    Let $X$ be a set of ordinals in $N$. Then there is a finite set $c\s A$, such that $X\in V[G_c]$.
\end{lemma}

\begin{lemma} 
    $N$ is a cardinal preserving outer model of $V$.
\end{lemma}
\begin{proof}
    Suppose, aiming for a contradiction, that there is a cardinal $\delta$ such that $N\models``\delta\text{ is not a cardinal''}$. Working in $N$, let $f\colon\gamma\rightarrow\delta$ be a witness for $|\delta|^N=\gamma<\delta$. By  \cref{lem: capturing}, $f\in V[G_c]$, for some finite $c=\{\kappa_0,\dots,\kappa_n\}\s A$ with $\kappa_0<\cdots<\kappa_n$. But this contradicts the fact that $V$ and $V[G_c]$ have the same class of cardinals by  \cref{lem: same cardinals}.
\end{proof}
\begin{lemma}
$\cof{\kappa}^N=\omega$ for all $\kappa\in \mcl M$.
\end{lemma}
\begin{proof}
Let  $\kappa\in A$. $V[G_\kappa]\models``\cof{\kappa}=\omega\text{''}$, as witnessed by a sequence $\vec{\kappa}=\langle\kappa_n\mid n<\omega\rangle\in V[G_\kappa]$ definable from $G_\kappa$. Since $\vec{\kappa}\in N$ by \cref{finite_piece_included}, the conclusion follows. 
\end{proof} 
\begin{proposition}\label{no-measure-below}
There is no measurable cardinal below $\Theta$ in $N$.
\end{proposition}
\begin{proof}
    It follows from the fact that $\mathbb{P}$ destroys the measurability of each $\kappa\in A$. Concretely, suppose, aiming for a contradiction, that there is $\lambda<\Theta^N$ such that $N\models``\lambda\text{ is a measurable cardinal''}$. Since $\Theta^N=\Theta^V$ by  \cref{lem: Theta computed correctly} and by $\ad_{\mathbb{R}}$ every regular cardinal below $\Theta$ in $V$ is measurable, it must be the case that $\lambda\in A$. Therefore, $\mathbb{P}_{\lambda}\Vdash``\cof{\lambda}=\omega\text{''}$. But $V[G_\lambda]\s N$ by  \cref{finite_piece_included} and so $N\models``\cof{\lambda}=\omega\text{''}$, contradicting the fact that $\lambda$ is measurable in $N$.
\end{proof}
Now in $N$, every uncountable cardinal below $\Theta$ is singular and nonmeasurable, we want to show that $\Theta$ is strongly regular and measurable in $N$. These are enough to conclude that $\Theta$ is the minimal uncountable regular cardinal, which is strongly regular and measurable. Hence \cref{mainthm}.

\begin{proposition}
    $\Theta$ is strongly regular in $N$.
\end{proposition}

\begin{proof}
Suppose $f \in N$ is a function from $(\mathcal{P}(\alpha))^N $ into $\Theta$ for some $\alpha < \Theta$. We will show that $f$ is bounded in $\Theta$. First we show that there is a finite $c \subset \mathcal{M} \setminus \alpha +1 $ such that $f$ is in $V[\{ G_\beta : \beta \in \mathcal{M} \cap (\alpha +1 \cup c) \} ]$.
 Let $\dot{f}$ be a $\mathbb{P}$-name for $f$ and let $p$ be a condition with support $d$ such that $p \Vdash "\dot{f} \text{ is a function from} (\mathcal{P}(\alpha))^N \text{ to } \check{\Theta}"$. Let $d' = d \cap \mathcal{M} \setminus \alpha +1$. We will identify conditions in $\mathbb P$ with their nontrivial components.

 By \cref{lem: capturing} and the proof of \cref{lem: same cardinals} every subset of $\alpha$ in $N$ is in \linebreak $V[\{ G_\beta : \beta \in \mathcal{M} \cap \alpha +1 \}]$.
For every finite $c \subset \mathcal{M} \cap \alpha +1$ and every $\mathbb{P}_c$-name $\tau$ for a subset of $\alpha$, there is a canonical name $\tau^*$ consisting of all pairs $(\check{\beta}, p)$ such that $\beta < \alpha$, $p \in \mathbb{P}_c$, and $p \Vdash \check{\beta} \in \tau$.
For such $\tau$ and finite $e \subset \mathcal{M} \setminus \alpha +1$, we let $\tau_{c,e}^*$ be the $\mathbb P_{c \cup e}$-name $\{ (\check{\beta}, p) : (\check{\beta}, p \rest c ) \in \tau^*\} $. 
Clearly for every $\beta < \alpha$, every $p \in \mathbb{P}_{c \cup e}$, and every $\mathbb P_c$-name $\tau$ for a subset of $\alpha$, we have $p \Vdash \check{\beta} \in \tau_{c,e}^*$ if and only $p \rest c \Vdash \check{\beta} \in \tau$. 
Hence, for every finite $e \subset \mathcal{M} \setminus \alpha +1$,  we have 
\[\mathcal{P}(\alpha) \cap N = \bigcup_{c \in [\mathcal{M} \cap \alpha +1]^{< \omega}}(\dot{A}_c)_{G_{c \cup e}},\]
where 
\begin{align*}
    \dot{A}_c = \{ (\tau_{c,e}^*, p) : & \, \tau \text{ a }\mathbb{P}_c\text{-name for a subset of } \alpha  \land p \in \mathbb{P}_{c \cup e} \}.
\end{align*}

Now we prove the claim made in the first paragraph. For all finite $c \subset \mathcal{M} \cap \alpha+1$, $\gamma < \Theta$, and $\mathbb{P}_c$-name $\tau$ for a subset of $\alpha$, let $\pi^{c,\tau,\gamma}$ be a $\mathbb P_{c \cup d'}$-name such that for any $\mathbb P_{c \cup d'}$-generic $H$, $(\pi^{c,\tau, \gamma})_H = ((\tau_{c,d'}^*)_{H}, \gamma)$. Let 
\begin{align*} 
\dot{h} = \{ (\pi^{c,\tau, \gamma}, q):  & \, c \subset \mathcal{M} \cap \alpha +1 \text{ finite } \land \tau \text{ a }\mathbb{P}_c\text{-name for a subset of } \alpha \\
& \, \land q \leq p \text{ with support } c \cup d' \text{ is such that } q \Vdash \dot{f}( \tau_{c,d'}^*) = \check{\gamma} \}.
\end{align*}
Arguing as in the proof of \cite[Lemma 2.1]{Apter_AD_pattern}, we can show that $p \Vdash \dot{f} = \dot{h}$. Consequently, $f \in V[\{ G_\beta : \beta \in \mathcal{M} \cap (\alpha +1 \cup d')\}]$.

Now we show that $f$ is bounded.
For any finite $c= \{ \kappa_1< \dots< \kappa_n\} \subset \mathcal{M}$, and any $(s_1, \dots, s_n)$ with $s_i \in [\kappa_i]^{< \omega}$, $i \leq n$, the conditions $((s_1, D_1), \dots, (s_n,D_n))$ and $((s_1, D_1'), \dots, (s_n,D_n'))$ of $\mathbb{P}_c$ are compatible for any $D_i, D_i' \in \mu_{\kappa_i}$ such that $D_i, D_i' \subset \kappa_i \setminus (\max(s_i) +1)$, $i \leq n$.
Hence, for any finite $c = \{\kappa_1, \dots, \kappa_n\} \subset \mathcal{M} \cap \alpha +1$ and $\mathbb P_c$-name $\tau$ for a subset of $\alpha$, the canonical name $\tau^*$ can be coded as 
\begin{align*}
\{ (\beta, (s_1, \dots, s_n)) : & \, \beta < \alpha \land (\bigwedge_{i \leq n} s_i \in [\kappa_i]^{< \omega}) \land \\
& \exists D_1, \dots, D_n \, ((\bigwedge_{i \leq n} D_i \in \mu_{\kappa_i}) \land ((s_1, D_1), \dots, (s_n,D_n) ) \Vdash \check{\beta} \in \tau )\}.
\end{align*}
Consequently, each $\tau^*$ can be coded as subset of $\alpha$. Every condition of $\mathbb P_{c \cup d'}$ for finite $c \subset \mathcal{M} \cap \alpha+1$ can de coded as a subset of $\xi := \max(d')$. 
Hence, we can define in $V$ a function $g: \mathcal{P}(\xi) \to \Theta$ by $g(B) = \gamma$ if $B$ codes a pair $(\tau^*,p)$ such that there are a finite $c \subset \mathcal{M} \cap \alpha +1$, a $\mathbb P_c$-name $\tau$ for a subset of $\alpha$, and $p \in \mathbb P_{c \cup d'}$ such that $p \Vdash \dot{f} (\tau_{c,d'}^*) = \check{\gamma}$. Otherwise, we let $g(B) = 0$. If $f$ is unbounded in $\Theta$, then $g: \mathcal{P}(\xi) \to \Theta$ is unbounded in $\Theta$. But $\mathrm{AD}$ implies that there is a surjection from $\mathbb R$ onto $P(\xi)$. Hence, if $f$ is unbounded, then there is in $V$ an unbounded map from $\mathbb R$ into $\Theta$, which contradicts the assumption $\Theta_{\sf meas}$.
\end{proof}

\begin{proposition}
Suppose that  $\mu$ is an $\mathbb{R}$-complete normal ultrafilter on $\Theta$. Then $\mu$ lifts up to $N$. In particular, $\Theta$ is measurable in $N$. 
\end{proposition}
\begin{proof}
    Working in $N$, define $\mu^\ast\df\{x\s\Theta\mid\exists y\in\mu\, (y\s x)\}$. We claim that $\mu^\ast$ witnesses the measurability of $\Theta$. Let us first show that $\mu^*$ is an ultrafilter. So let $x\in\mathcal{P}(\Theta)^N$ and suppose $x\notin\mu^\ast$.     By  \cref{lem: capturing}, $x\in V[G_c]$ for some finite $c\s \mcl M$.
    Suppose that $c\df\{k_0,\dots,\kappa_n\}$ where $\kappa_0<\cdots<\kappa_n$. Then \cref{lifting_a_measure} yields that \[
    V[G_{\kappa_0}]\models``\nu_0\df\{x\s\Theta\mid\exists y\in\mu\, (y\s x)\}\text{ is a measure on $\Theta$''}.
    \]
    Working in $V[G_{\kappa_0}]$, we first apply \cref{lifting_a_measure} to lift $\mu_{\kappa_1}$ to $\mu_{1}^*\df \mu_{\kappa_1}^{*G_{\kappa_0}}$,  and applying once again  \cref{lifting_a_measure} to $\kappa_1$, $\mbb P_{\mu_{1}^*}$ and $\Theta$. Note that since then   forcing $\mbb P_{\kappa_1}$ is dense in $\mbb P_{\mu_{1}^*}$,  we get \[
    V[G_{\kappa_0}\times G_{\kappa_1}]\models``\nu_1\df\{x\s\Theta\mid\exists y\in\mu\, (y\s x)\}\text{ is a measure on $\Theta$''}.
    \]
    Carrying out the above argument $n+1$ many times, we have that
     
    \begin{equation}\label{eq: measure_lifted}
        V[G_c]\models``\nu_c\df\{x\s\Theta\mid\exists y\in\mu\, (y\s x)\}\text{ is a measure on $\Theta$''}.
    \end{equation} 
 As $x\notin\mu^\ast$, it has to be the case that $x\notin\nu_c$. Therefore, $\Theta\setminus x\in\nu_c$ and so $\Theta\setminus x\in\mu^\ast$, showing that $\mu^\ast$ is an ultrafilter. 

    Next we deal with $\Theta$-completeness. Suppose towards a contradiction that there exists a sequence $\langle x_\alpha\mid\alpha<\beta\rangle\in N$ of sets in $\mu^\ast$ with $\beta<\Theta$ such that $\bigcap_{\alpha<\beta}x_\alpha\notin\mu^\ast$. As  above, we use  \cref{lem: capturing} to infer the existence of a finite set $c\s A$ such that $\bigcap_{\alpha<\beta}x_\alpha\in V[G_c]$. The previous paragraph shows that $\nu_c$, as defined in \cref{eq: measure_lifted}, is a measure on $\Theta$ in $V[G_c]$. By definition, $\bigcap_{\alpha<\beta}x_\alpha\notin\nu_c$, and so $\Theta\setminus\bigcap_{\alpha<\beta}x_\alpha\in\nu_c$. This means that there is $Y\in \mu$ with \[
    Y\s\Theta\setminus\bigcap_{\alpha<\beta}x_\alpha=\bigcup_{\alpha<\beta}(\Theta\setminus x_\alpha).
    \]
    Define, in $N$, the map $f:Y\rightarrow\beta$  by $f(y)\df\min\{\alpha<\beta\mid y\notin x_\alpha\}$. Clearly, $f$ can be coded as a set of ordinals and so we may appeal to  \cref{lem: capturing} to deduce the existence of a finite set $d$ such that $c\s d\s \mcl M$ and $f\in V[G_d]$. We have already argued that $V[G_d]\models``\nu_d\df\{x\s\Theta\mid\exists y\in\mu\, (y\s x)\}\text{ is a measure on $\Theta$''}$. Working in $V[G_d]$, one can easily show that there is $\alpha<\beta$ with $f^{-1}(\alpha)\df\{y\in Y\mid f(y)=\alpha\}\in\nu_d$. If not $Y\setminus f^{-1}(\alpha)=\{y\in Y\mid f(y)\neq \alpha\}\in \nu_d$, for all $\alpha<\beta$. By $\Theta$-completeness of $\nu_d$\footnote{It is easily seen that $\mbb R$-completeness implies $\Theta$-completeness.}, we have $\bigcap_{\alpha<\beta}(Y\setminus f^{-1}(\alpha))\in\nu_d$, contradicting the fact that $\bigcap_{\alpha<\beta}(Y\setminus f^{-1}(\alpha))=\emptyset$. So let $\alpha<\beta$ be such that $f^{-1}(\alpha)\in\nu_d$. Thus, there is $Z\s f^{-1}(\alpha)$ with $Z\in\mu$, and so $Z\s \Theta\setminus x_\alpha$. By definition, we get $\Theta\setminus x_\alpha\in\mu^\ast$, but this is impossible since its complement, i.e. $x_\alpha$, is also in $\mu^\ast$. 

    It remains to show that $N\models``\mu^\ast\text{ is normal''}$. If $$N\models``g:\Theta\rightarrow\Theta\text{ is a regressive function'',}$$ then as $g$ may be coded by a set of ordinals in $N$, for some finite $c\s \mcl M$ we have $g\in V[G_c]$. Since $V[G_c]\models``\nu_c\text{ is a measure on $\Theta$''}$, it must be the case that $V[G_c]\models``g\text{ is constant on some $X\in\nu_c$''}$. In particular, $N\models``g\restriction X\text{ is constant and $X\in\mu^\ast$''}$.
\end{proof}

\section{Open Questions}
In \cite{gitik2024first}, the authors proved that if $\kappa$ is a strongly inaccessible non-Mahlo measurable cardinal, then $K\models o(\kappa)\geq\kappa+1$. Accordingly, they wonder whether the lower bound can be improved. In light of our main theorem, we ask the following.  
\begin{question}\label{q1}
What is the exact consistency strength of $``\Theta$ is the least inaccessible and the least measurable cardinal''?
\end{question}
The Prikry-type construction carried out to get the configuration described in our main theorem seems to be flexible enough to manipulate the measure on $\Theta$. For instance, the following are open.
\begin{question}\label{q2}
    Assume $\zf+\ad+\dc_{\mathbb R}$.
    \begin{enumerate}
        \item Can the least inaccessible be the least measurable that carries a normal measure concentrating on points of cofinality $\omega_1$?
        \item Can the least inaccessible be the least measurable that carries a normal measure concentrating on regular cardinals?
    \end{enumerate}
\end{question}
On the other hand, one may look at the structure of ulfrafilters on $\Theta$.
\begin{question}\label{q3}
    Assume $\Theta$ is inaccessible and the least measurable cardinal. How large can $o(\Theta)$ be?
\end{question}

\begin{question}\label{q4}

Given an ordinal $\xi<\Theta$.
    Is it consistent that $\Theta$ is the least measurable and the least strongly regular, but the $\xi$-th uncountable regular cardinal?
\end{question}

The authors would like to thank Karagila and Chan for their valuable responses to their questions. In particular, Karagila addressed successfully Questions \ref{q1} and \ref{q4}, while Chan remarked that Question \ref{q2} has a positive answer under $\lsa$.

We conclude this note by proposing a couple of questions along the lines of celebrated Gitik's result on the consistency of $\zf+``$all infinite cardinals have countable cofinality" \cite{gitik1980all}. Observe that the rank initial segment of $N$ up to $\Theta$ is not a model of $\zf$ so it is still open whether the large cardinal assumption of Gitik's theorem can be improved.
\begin{question}
    What is the exact consistency strength of $\zf+``$all infinite cardinals have countable cofinality".
\end{question}
\begin{question}
    Is $\zf+\dc+``\text{all infinite cardinals are singular"}$ consistent?
\end{question}

\emph{\small Acknowledgments.
The authors would like to thank Grigor Sargsyan for useful discussions. This work was partially supported by the Simons Foundation grant (award no. SFI-MPS-T-Institutes-00010825) and from State Treasury funds as part of a task commissioned by the Minister of Science and Higher Education under the project “Organization of the Simons Semesters at the Banach Center - New Energies in 2026-2028” (agreement no. MNiSW/2025/DAP/491). The first and second author's work was funded by National Science Center, Poland under MAESTRO 15 call, registration number UMO-2023/50/A/ST1/00258 The third author’s work is funded by the National Science Center, Poland under the Weave-Unisono Call, registration number UMO-2023/05/Y/ST1/00194.}
\begin{center}   
        \includegraphics[width=0.8\linewidth]{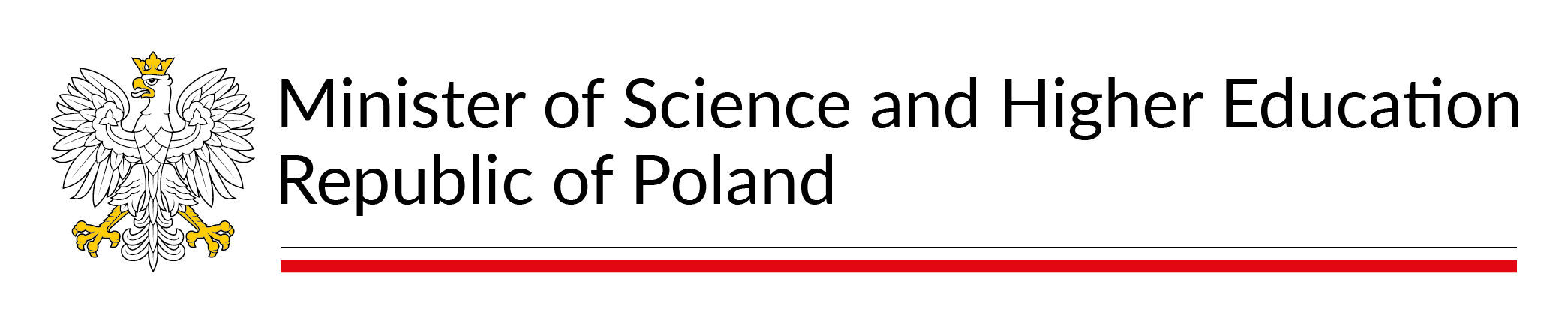}
    \end{center}

    \begin{center}   
        \includegraphics[width=0.8\linewidth]{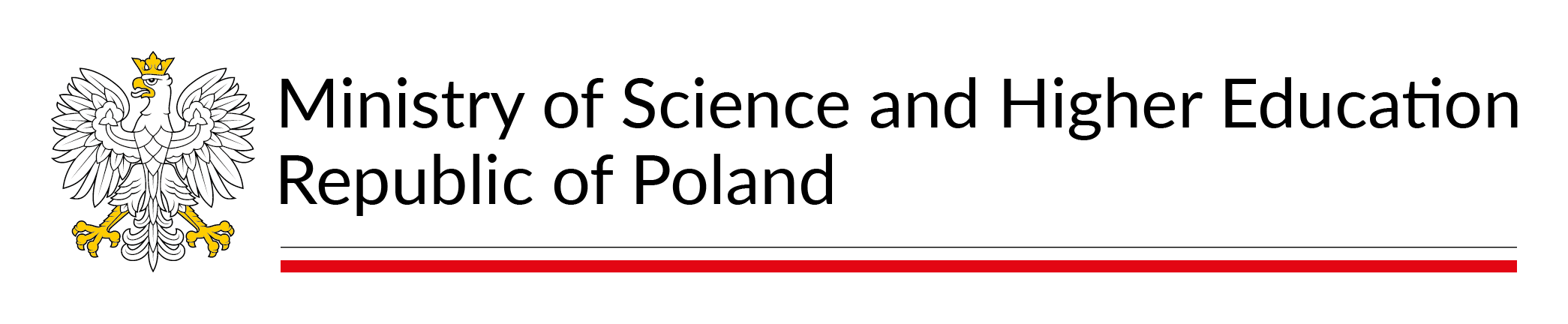}
    \end{center}
\printbibliography[heading=bibintoc]
\end{document}